\documentclass[11pt]{amsart}
\usepackage{amsmath}
\usepackage{amssymb}
\usepackage{latexsym}
\usepackage{color,fancyhdr,lastpage,graphicx}
\theoremstyle{plain}

\newtheorem{lemma}{Lemma}[section]

\newtheorem{theorem}[lemma]{Theorem}
\newtheorem{proposition}[lemma]{Proposition}

\theoremstyle{remark}

\def\bg{\begin{color}{green}}
\def\br{\begin{color}{red}}
\def\eg{\end{color}}
\def\er{\end{color}}
\def\PP{\mathbb{P}}

\def\le{\leqslant}
\def\pd{\partial}
\def\ds{\pd_s}
\def\dt{\pd_t}
\def\du{\pd_u}
\def\dr{\pd_r}

\def\ge{\geqslant}

\def\es{\emptyset}
\def\da{\downarrow}

\def\E{{\mathbb E}}
\def\O{{\Omega}}
\def\R{{\mathbb R}}

\def\N{{\mathbb N}}

\def\a{{\alpha}}

\def\d{{\delta}}

\def\g{{\gamma}}
\def\G{{\Gamma}}
\def\s{{\sigma}}

\def\o{{\omega}}
\def\z{{\zeta}}

\def\cF{{\mathcal F}}

\def\cD{{\mathcal D}}
\def\cW{{\mathcal W}}

\def\q{\quad}

\def\<{\langle}
\def\>{\rangle}

\def\ua{\uparrow}
\def\sse{\subseteq}
\def\sm{\setminus}
\title{Two-parameter stochastic calculus and {M}alliavin's integration-by-parts formula on {W}iener space }
\author{J.R. Norris}
\address{University of Cambridge\\
Statistical Laboratory, Centre for Mathematical Sciences, Wilberforce Road, Cambridge, CB3 0WB, UK}
\email{j.r.norris@statslab.cam.ac.uk}

\begin{document}
\bibliographystyle{plain}

\begin{abstract}
The integration-by-parts formula discovered by Malliavin for the It\^ o map
on Wiener space is proved using the two-parameter stochastic calculus. It is
also shown that the solution of a one-parameter stochastic differential
equation driven by a two-parameter semimartingale is itself a 
two-parameter semimartingale.
\end{abstract}

\maketitle


\section{Introduction}
The stochastic calculus of variations was conceived by Malliavin \cite{MR517243,
MR517250,MR536013} as
follows. Let $(z_t)_{t\ge0}$ denote the Ornstein--Uhlenbeck process on Wiener
space $(W,\cW,\mu)$ and let $\Phi:W\to\R^d$ denote the (almost-everywhere
unique) It\^ o map obtained by solving a stochastic differential equation
in $\R^d$ up to time $1$. Then $(z_t)_{t\ge0}$
is stationary and reversible, so, for 
functions $f,g$ on $\R^d$, setting $F=f\circ\Phi,G=g\circ\Phi$,
\begin{equation}\label{REV}
\E\left[\{F(z_t)-F(z_0)\}\{G(z_t)-G(z_0)\}\right]
=-2\E\left[F(z_0)\{G(z_t)-G(z_0)\}\right].
\end{equation}
Once certain terms of mean zero are subtracted, a
differentiation of this identity with respect to $t$ {\em inside the expectation} is possible, 
and leads to the integration-by-parts formula on Wiener space
\begin{equation}\label{IBP}
\int_W\nabla_i f(\Phi)\G^{ij}\nabla_j g(\Phi)d\mu=-\int_W f(\Phi)LGd\mu,
\end{equation}
where $LG$ and the {\em covariance matrix} $\G$ will be defined below.
As is now well known, this formula and its generalizations hold the key 
to many deep results of stochastic analysis.

Malliavin's proof of the integration-by-parts formula was based on a
{\em transfer principle}, allowing some calculations for two-parameter
random processes to be made using classical differential calculus.
Stroock \cite{MR642917,MR603973,MR616961}
and Shigekawa \cite{MR582167}
gave alternative derivations having a a more functional-analytic 
flavour. Bismut \cite{MR621660} gave another derivation based on the 
Cameron--Martin--Girsanov formula. Elliott and Kohlmann \cite{MR972781}
and Elworthy and Li \cite{MR1297021}
found further elementary approaches to the formula.
The alternative proofs are
relatively straightforward. Nevertheless, we have found it interesting
to go back to Malliavin's original approach in \cite{MR536013} and to review the calculations
needed, especially since this can be done now in a more explicit way 
using the two-parameter stochastic calculus, as formulated in \cite{MR1347353}.

In Section \ref{OU} we review in greater detail the various mathematical
objects mentioned above. Then, in Section \ref{TS}, we review some
points of two-parameter stochastic calculus from \cite{MR1347353}.
Section \ref{NRR} contains the main technical result of the paper, which is a regularity property for
two-parameter stochastic differential equations. We consider
equations in which some components are given by two-parameter
integrals and others by one-parameter integrals. It is shown, under suitable
hypotheses, that the components which are presented as one-parameter integrals are 
in fact two-parameter semimartingales. This is useful because one can then compute martingale properties
for both parameters by stochastic calculus. The sorts of differential equation to which this
theory applies are just one way to realise continuous random processes indexed by the plane. See
the survey \cite{MR2216962} by L\' eandre for a wider discussion. 
But this regularity property makes our processes more tractable to analyse than some others.
This is illustrated in Section \ref{C}, where we do the calculations
needed to obtain the integration-by-parts formula.

\section{Integration-by-parts formula}\label{OU}
The Wiener space $(W,\cW,\mu)$ over $\R^m$ is a probability space
with underlying set $W=C([0,\infty),\R^m)$, the set of continuous paths in $\R^m$. 
Let $\cW^o$ denote the 
$\s$-algebra on $W$
generated by the family of coordinate functions $w\mapsto w_s:W\to\R^m$, 
$s\ge0$, and let $\mu^o$ be Wiener measure on $\cW^o$, that is to say,
the law of a Brownian motion in $\R^m$ starting from $0$.
Then $(W,\cW,\mu)$ is the completion of the probability space
$(W,\cW^o,\mu^o)$.
Write $\cW_s$ for the $\mu$-completion of $\s(w\mapsto w_r:r\le s)$.
Let $X_0,X_1,\dots,X_m$ be vector fields on $\R^d$, with bounded
derivatives of all orders. Fix $x_0\in\R^d$ and consider the 
stochastic differential equation
$$
\pd x_s=X_i(x_s)\pd w^i_s+X_0(x_s)\pd s.
$$
Here and below, the index $i$ is summed from $1$ to $m$, and $\pd$ denotes the
Stratonovich differential. There exists a map
$x:[0,\infty)\times W\to\R^d$ with the following properties:
\begin{itemize}
\item $x$ is a continuous semimartingale on $(W,\cW,(\cW_s)_{s\ge0},\mu)$,
\item for $\mu$-almost all $w\in W$, for all $s\ge0$ we have
$$
x_s(w)=x_0+\int_0^s X_i(x_r(w))\pd w^i_r+\int_0^s X_0(x_r(w))dr.
$$
\end{itemize}
The first integral in this equation is the Stratonovich stochastic integral.
Moreover, for any other
such map $x'$, we have $x_s(w)=x'_s(w)$ for all $s\ge0$, for $\mu$-almost all
$w$.  We have chosen here a Stratonovich rather than an It\^ o formulation 
to be consistent with later sections, where we have made this choice
in order to
take advantage of the simpler calculations which the 
Stratonovich calculus allows.
The It\^ o map referred to above is the map $\Phi(w)=x_1(w)$.

We can define on some complete probability space,
$(\O,\cF,\PP)$ say, a two-parameter,
continuous, zero-mean Gaussian field $(z_{st}:s,t\ge0)$ with values in
$\R^m$, and with covariances given by
$$
\E(z_{st}^iz_{s't'}^j)=\d^{ij}(s\wedge s')e^{-|t-t'|/2}.
$$
Such a field is called an Ornstein--Uhlenbeck sheet. Set $z_t=(z_{st}:s\ge0)$.
Then, for $t>0$, both $z_0$ and $z_t$ are Brownian motions in $\R^m$ and
$(z_0,z_t)$ and $(z_t,z_0)$ have the same distribution.  We have now defined
all the terms in, and have justified, the identity (\ref{REV}).

Consider the following stochastic differential equation for an 
unknown process $(U_s:s\ge0)$ in the space of $d\times d$ matrices
$$
\pd U_s=\nabla X_i(x_s)U_s\pd w_s^i+\nabla X_0(x_s)U_s\pd s,\q U_0=I.
$$
This equation may be solved, jointly with the equation for $x$, in exactly
the same sense as the equation for $x$ alone. Thus we obtain a map
$U:[0,\infty)\times W\to \R^d\otimes(\R^d)^*$, with properties
analogous to those of $x$. Moreover, by solving
an equation for the inverse, we can see 
that $U_s(w)$ remains invertible for all $s\ge0$,
for almost all $w$. Write $U^*_s$ for the transpose matrix and set 
$\G_s=U_sC_sU_s^*$, where 
$$
C_s=\int_0^sU^{-1}_rX_i(x_r)\otimes U^{-1}_rX_i(x_r)dr.
$$
Set also 
\begin{align*}
L_s=-U_s\int_0^sU^{-1}_rX_i(x_r)\pd w_r^i
&+U_s\int_0^sU^{-1}_r\{\nabla^2X_i(x_r)\pd w^i_r+\nabla^2X_0(x_r)dr\}\G_r,\\
&+U_s\int_0^sU^{-1}_r\nabla X_i(x_r)X_i(x_r)dr
\end{align*}
and define for $G=g\circ\Phi$
$$
LG=L_1^i\nabla_ig(x_1)+\G_1^{ij}\nabla_i\nabla_jg(x_1).
$$
We have now defined all the terms appearing in the integration-by-parts 
formula (\ref{IBP}). We will give a proof in Section \ref{C}.

\section{Review of two-parameter stochastic calculus}\label{TS}

In \cite{MR1347353}, building on the fundamental works
of Cairoli and Walsh \cite{MR0420845} 
and Wong and Zakai \cite{MR0370758,MR0651571}, we gave an account of 
two-parameter stochastic calculus, suitable for the development
of a general theory of two-parameter hyperbolic stochastic differential
equations. We recall here, for the reader's convenience, the main
features of this account.

We take as our probability space $(\O,\cF,\PP)$ the canonical complete 
probability space of an $m$-dimensional Brownian sheet $(w_{st}:s,t\ge0)$,
extended to a process $(w_{st}:s,t\in\R)$ 
by independent copies in the other three quadrants.
Thus $w_{st}=(w^1_{st},\dots,w^m_{st})$ is a continuous, zero-mean Gaussian 
process, with covariances given by
$$
\E(w_{st}^iw_{s't'}^j)=\d^{ij}(s\wedge s')(t\wedge t'),\q
i,j=1,\dots,m,\q
s,t\ge0,\q s',t'\ge0.
$$
It will be convenient to define also $w^0_{st}=st$ for all $s,t\in\R$.
For $s,t\ge0$, write $\cF_{st}$ for the completion with respect to $\PP$
of the $\s$-algebra generated by $w_{ru}$ for $r\in(-\infty,s]$ and $u\in(-\infty,t]$.
We say that a two-parameter process $(x_{st}:s,t\ge0)$ is {\em adapted}
if $x_{st}$ is $\cF_{st}$-measurable for all $s,t\ge0$,
and is {\em continuous} 
if $(s,t)\mapsto x_{st}(\o)$ is continuous on $(\R^+)^2$
for all $\o\in\O$. The
{\em previsible} $\s$-algebra on $\O\times(\R^+)^2$
is that generated by sets of the form 
$A\times(s,s']\times(t,t']$ with $A\in\cF_{st}$.
If we allow $A\in\cF_{s\infty}$ in this definition, we get the {\em $s$-previsible} $\s$-algebra.

The classical approach to defining stochastic integrals,
by means of an isometry of  Hilbert spaces, adapts in a straightforward way
from one-dimensional times to two, allowing the construction of
stochastic integrals with respect to certain two-parameter processes,
in particular with respect to the Brownian sheet.
\def\rootnote{In 
the one-parameter theory, whilst It\^ o's original integrals were
defined with respect to Brownian motion, it was later found that this
could be generalized to define integrals with respect to 
all continuous martingales, relying on the properties of the
quadratic variation process. This general context had become accepted
as the right one for stochastic integrals when the two-parameter
theory was under development. However it did not prove possible to
generalize this in all respects to two-parameters. In retrospect, one can 
perhaps see that the generality possible in one-parameter is a freak,
possible only because a time-reparametrization brings the most general
quadratic variation back to that of Brownian motion.
The most satisfactory theory in two-parameters seems to us to take
as integrators some classes of processes which are themselves defined as
integrals with respect to the Brownian sheet.}
Given an $s$-previsible process\footnote{We
write any time parameter with respect to which a process is previsible, here $s$,
as a subscript. Where previsibility is not assumed, here in $t$, we write the parameter in parentheses.}
$(a_s(t):s,t\ge0)$, such that 
$$
\E\int_0^s\int_0^ta_r(u)^2drdu<\infty
$$
for all $s,t\ge0$, we can define, for $i=1,\dots,m$ and all
$t_1,t_2\ge0$ with $t_1\le t_2$,
one-parameter processes $M$ and $A$ by
\begin{equation}\label{II}
M_s=\int_0^s\int_{t_1}^{t_2}a_r(t)d_rd_tw_{rt}^i,\q
A_s=\int_0^s\int_{t_1}^{t_2}a_r(t)^2drdt.
\end{equation}
Then $M$ is a continuous $(\cF_{s\infty})_{s\ge0}$-martingale,
with quadratic variation process $[M]=A$.
A localization argument by adapted initial open sets (see below)
allows an extension of the integral under weaker integrability conditions. 
By the Burkholder--Davis--Gundy inequalities, for all $\a\in[2,\infty)$, there is a
constant $C(\a)<\infty$ such that
\begin{equation}\label{BDG}
\E\left(\left|\int_{s_1}^{s_2}\int_{t_1}^{t_2}a_s(t)d_sd_tw_{st}^i\right|^\a\right)
\le C(\a)\E\left(\left|\int_{s_1}^{s_2}\int_{t_1}^{t_2}a_s(t)^2dsdt\right|^{\a/2}\right).
\end{equation}

By an $(s,t)$-{\em semimartingale}, $s$-{\em semimartingale},
{\em $t$-semimartingale},
we mean, respectively, previsible
processes $(x_{st}:s,t\ge0)$, $(p_{st}:s,t\ge0)$, $(q_{st}:s,t\ge0)$
for which we may write
\begin{align*}
&x_{st}-x_{s0}-x_{0t}+x_{00}\\
&\q\q=\sum_{i=0}^m\int_0^s\int_0^t(x_{ru}'')_id_rd_uw_{ru}^i
+\sum_{i,j=0}^m\int_0^s\int_{-1}^t\left(\int_{-1}^s\int_0^t
(x_{ru}''(r',u'))_{ij}d_{r'}d_uw_{r'u}^j\right)d_rd_{u'}w^i_{ru'}\\
\end{align*}
and 
$$
p_{st}-p_{0t}
=\sum_{i=0}^m\int_0^s\int_{-1}^t(p_{rt}'(u'))_id_rd_{u'}w_{ru'}^i,\q
q_{st}-q_{s0}
=\sum_{i=0}^m\int_{-1}^s\int_0^t(q_{su}'(r'))_id_{r'}d_uw_{r'u}^i.
$$
Here, $(x_{st}'':s,t\ge0)$ is a previsible process, having
components $(x_{st}'')_i$, subject to
certain local integrability conditions, which are implied, in particular,
by almost sure local boundedness. 
The process $(x_{st}''(r,u):s,t\ge0, r,u\in\R)$ is required
to be previsible in $(\o,s,t)$ and (Borel) measurable in $(r,u)$, 
with $x_{st}''(r,u)=0$ for $r>s$ or $u>t$, and is subject to 
similar local integrability conditions.
The inner and outer parts
of the second integral are both cases of the stochastic integral at (\ref{II}), or its $t$-analogue,
or of the usual Lebesgue integral,
and the value of the iterated integral is unchanged if we reverse the order in which
the integrals are taken.
The integrals appearing in the expression for $x_{st}$ are called stochastic
integrals of the {\em first} and {\em second kind}.
The processes $(p_{st}'(u):s,t\ge0, u\in\R)$ 
and $(q_{st}'(r):s,t\ge0, r\in\R)$ are required
to be previsible in $(\o,s,t)$ and measurable in $u$ and
$r$, respectively, with $p_{st}'(u)=0$ for $u>t$ and $q_{st}'(r)=0$ for $r>s$,
and are subject to similar local integrability conditions.
For fixed $t\ge0$, if $(x_{s0}:s\ge0)$ is a continuous $(\cF_{s0})_{s\ge0}$-semimartingale,
then $(x_{st}:s\ge0)$ is a continuous $(\cF_{st})_{s\ge0}$-semimartingale,
in the usual one-parameter sense.
Also $(p_{st}:s\ge0)$ is a continuous $(\cF_{st})_{s\ge0}$-semimartingale, for all $t\ge0$.

The heuristic formulae
\begin{align*}
d_sd_tx_{st}
&=\sum_{i=0}^m(x_{st}'')_id_sd_tw_{st}^i
+\sum_{i,j=0}^m\int_{-1}^s\int_{-1}^t(x_{st}''(r,u))_{ij}
d_sd_uw_{su}^id_rd_tw^j_{rt},\\
d_sp_{st}
&=\sum_{i=0}^m\int_{-1}^t(p_{st}'(u))_id_sd_uw_{su}^i,\\
d_tq_{st}
&=\sum_{i=0}^m\int_{-1}^s(q_{st}'(r))_id_rd_tw_{rt}^i
\end{align*}
provide a good intuition in representing the two-parameter
increment 
$$
d_sd_tx_{st}=x_{s+ds,t+dt}-x_{s,t+dt}-x_{s+ds,t}+x_{st}
$$
and the one-parameter increments $d_sp_{st}=p_{s+ds,t}-p_{st}$
and $d_tq_{st}=q_{s,t+dt}-q_{st}$
in terms of a linear combinations of increments, and of 
products of increments of the Brownian sheet.

By a ({\em two-parameter}) {\em semimartingale}, we mean a process
which is at the same time an $(s,t)$-semimartingale, an
$s$-semimartingale and a $t$-semimartingale.
Such processes are necessarily continuous. 
An $(s,t)$-semimartingale which is constant on the $s$-axis and $t$-axis
is a semimartingale.
By an obvious choice of integrands, the process
$(w_{st}:s,t\ge0)$ is itself a semimartingale.
The choice of lower limit $-1$ is useful
to us in allowing as semimartingales a pair of independent
$\R^m$-valued Brownian motions $(z_{s0}:s\ge0)$ and
$(b_{0t}:t\ge0)$, given by
$$
z_{s0}=\int_0^s\int_{-1}^0d_rd_u w_{ru},\q
b_{0t}=\int_{-1}^0\int_0^td_rd_u w_{ru},
$$
which are moreover independent of $(w_{st}:s,t\ge0)$.
Here and below, we bring one-parameter processes defined on the $s$ or $t$ 
axes into the class of two-parameter processes by extending them as constant
in the second parameter. 

We say that a subset $\cD\sse(\R^+)^2$ is an {\em initial open set} if it is non-empty
and is a union of 
rectangles of the form $[0,s)\times[0,t)$, where $s,t\ge0$.
A random subset $\cD\sse\O\times(\R^+)^2$ is {\em adapted}
if the event $\{(s,t)\in\cD\}$ is $\cF_{st}$-measurable for all $s,t\ge0$.
For an adapted initial open set $\cD$, a process $(x_{st}:(s,t)\in\cD)$
is a {\em semimartingale in} $\cD$ if there exists a sequence of adapted initial
open sets $\cD_n\ua\cD$, almost surely, and a sequence of 
semimartingales $(x^n_{st}:s,t\ge0)$, such that $x_{st}=x^n_{st}$ 
for all $(s,t)\in\cD_n$ for all
$n$. The notion of an {\em $s$-semimartingale in} $\cD$ is defined analogously.
We write $\z(\cD)$ for the boundary of $\cD$ as a subset of $(\R^+)^2$.
In particular, if $\cD=(\R^+)^2$, then $\z(\cD)=\es$.

The theory which we now describe is symmetrical in $s$ and $t$. Where 
a statement is made for $s$, there is also a corresponding statement for $t$,
which we shall often omit.
Let $(x_{st}:s,t\ge0)$ and $(x'_{st}:s,t\ge0)$ be $s$-semimartingales
and let $(a_{st}:s,t\ge0)$ be a locally bounded previsible process, for example, a continuous
adapted process.
There exist $s$-semimartingales which, for each $t\ge0$, provide
versions of the one-parameter stochastic integral and the 
one-parameter covariation process
$$
\z^1_{st}=\int_0^sa_{rt}d_rx_{rt},\q
\z^2_{st}=\int_0^sd_rx_{rt}d_rx'_{rt}.
$$
From now on, when we write these integrals, we assume that such a 
version has been chosen.
We define also four types of two-parameter integral, each of which is
a (two-parameter) semimartingale. These are written
$$
\z_{st}^3=\int_0^s\int_0^ta_{ru}d_rd_ux_{ru},\q
\z_{st}^4=\int_0^s\int_0^td_rx_{ru}d_uy_{ru},
$$
$$
\z_{st}^5=\int_0^s\int_0^td_rx_{ru}d_rd_uy_{ru},\q
\z_{st}^6=\int_0^s\int_0^td_rd_ux_{ru}d_rd_uy_{ru}.
$$
In the first and last integral, we require $x$ to be an $(s,t)$-semimartingale,
whereas, in the second and third, $x$ should be an 
$s$-semimartingale. We require that $y$ be a $t$-semimartingale in the
second integral and an $(s,t)$-semimartingale in the third and fourth. 
All these integrals  are defined as sums of certain integrals of the first 
and second kind with respect to the Brownian sheet. 
We refer to \cite{MR1347353} for the details. 
We use the following differential notations:
\begin{align*}
d_sz_{st}&=a_{st}d_sx_{st} &&\text{means}& z_{st}-z_{0t}&=\z_{st}^1,\\
d_sz_{st}&=d_sx_{st}d_sx'_{st} &&\text{means}& z_{st}-z_{0t}&=\z_{st}^2,\\
d_sd_tz_{st}&=a_{st}d_sd_tx_{st} &&\text{means}& z_{st}-z_{s0}-z_{0t}+z_{00}&=\z_{st}^3,\\
d_sd_tz_{st}&=d_sx_{st}d_ty_{st} &&\text{means}& z_{st}-z_{s0}-z_{0t}+z_{00}&=\z_{st}^4,\\
d_sd_tz_{st}&=d_sx_{st}d_sd_ty_{st} &&\text{means}& z_{st}-z_{s0}-z_{0t}+z_{00}&=\z_{st}^5,\\
d_sd_tz_{st}&=d_sd_tx_{st}d_sd_ty_{st} &&\text{means}& z_{st}-z_{s0}-z_{0t}+z_{00}&=\z_{st}^6.
\end{align*}
The integrals $\z_{st}^2$, $\z_{st}^5$ and $\z_{st}^6$ all vanish if $d_sx_{st}=a_{st}ds$.
It is shown in \cite{MR1347353} that a series of identities hold among
the various types of integral, which can be expressed conveniently in terms of
this differential notation. Some identities assert the associativity of
products involving a combination of three
differentials or processes, the others are written as the 
following three rules
\begin{align*}
d_s(f(x_{st}))&=f'(x_{st})d_sx_{st}+\tfrac12 f''(x_{st})d_sx_{st}d_sx_{st},\\
d_s(a_{st}d_tx_{st})&=d_sa_{st}d_tx_{st}+a_{st}d_sd_tx_{st}
+d_sa_{st}d_sd_tx_{st},\\
d_s(d_tx_{st}d_ty_{st})&=d_sd_tx_{st}d_ty_{st}+d_tx_{st}d_sd_ty_{st}
+d_sd_tx_{st}d_sd_ty_{st}.
\end{align*}
These rules combine the usual calculus of partial 
differentials with It\^ o calculus in an obvious way. 
As a consequence, we can obtain a geometrically simpler 
Stratonovich-type calculus by defining, for processes $(x_{st}:s,t\ge0)$
and $(y_{st}:s,t\ge0)$, some further integrals, corresponding
to the following differential rules
$$
X_{st}\ds X_{st} =X_{st}dY_{st}+\tfrac12d_sX_{st}d_sY_{st},\q
\ds X_{st}\ds Y_{st}=\ds X_{st} d_s Y_{st}=d_sX_{st}d_sY_{st},
$$
where $X_{st}$ may stand for any one of $x_{st},d_tx_{st}$,$\dt x_{st}$
and $Y_{st}$ may stand for any one of $y_{st},d_ty_{st}$,$\dt y_{st}$.
Then we have
\begin{align*}
\ds(f(x_{st}))&=f'(x_{st})\ds x_{st},\\
\ds (a_{st}\dt x_{st})&=\ds a_{st}\dt x_{st}+a_{st}\ds\dt x_{st},\\
\ds (\dt x_{st}\dt y_{st})&=\ds\dt x_{st}\dt y_{st}+\dt x_{st}\ds\dt y_{st}.
\end{align*}
The Brownian sheet $(w_{st}:s,t\ge0)$ and the boundary Brownian
motions $(z_{s0}:s\ge0)$ and $(b_{0t}:t\ge0)$ have some special
properties, which are reflected in the following differential
formulae, for $1\le i,j\le m$,
$$
d_sd_tw^i_{st}d_sd_tw_{st}^j=\d^{ij}dsdt,\q
d_sz_{s0}^id_sz_{s0}^j=\d^{ij}ds,\q
d_tb_{0t}^id_tb_{0t}^j=\d^{ij}dt,
$$
and, for any semimartingale $(x_{st}:s,t\ge0)$,
$$
d_sx_{st}d_sd_tw^i_{st}=d_tx_{st}d_sd_tw_{st}^i=0.
$$

\section{A regularity result for two-parameter stochastic differential
equations}\label{NRR}
We discussed in  \cite{MR1347353} a class of
two-parameter hyperbolic stochastic differential equations,
in which there is given, for a system
of processes $(x_{st},p_{st},q_{st}:s,t\ge0)$, one equation for the
mixed second-order differential $d_sd_tx_{st}$, together with two further
equations for the one-parameter differentials $d_sp_{st}$ and $d_tq_{st}$.
We review briefly the details below, and then give
a new regularity result, which we need for our application to Malliavin's
integration-by-parts formula, but which may be of independent interest.
This result concerns 
the process $(p_{st}:s,t\ge0)$ (and analogously also
$(q_{st}:s,t\ge0)$), which, since integrated in $s$, has naturally
the regularity of an $s$-semimartingale. The point at issue is 
whether $(p_{st}:s,t\ge0)$ is a full (two-parameter) semimartingale.
A method to establish this is
stated in \cite[pp. 299, 315-316]{MR1347353}, but the argument given is 
incomplete. A full proof is given below in Theorem \ref{TPS}.
As an illustrative example, we note that, if $(w_{st}:s,t\ge0)$
is a Brownian sheet with values in $\R^m$, then the result will show that
there is a two-parameter semimartingale $(x_{st}:s,t\ge0)$ such that,
for all $t\ge0$, the process $(x_{st}:s\ge0)$ satisfies the one-parameter
stochastic differential equation
$$
\ds x_{st}=X_i(x_{st})\ds w^i_{st}+X_0(x_{st})\pd s,
$$
with given initial values $x_{0t}=x_0$, say. This is useful because, now, despite 
the irregular dependence of the Brownian sheet on $t$, we can use a differential
calculus in $t$ as well as in $s$.

Consider the class of hyperbolic stochastic differential 
equations in $(\R^+)^2$ of the form
\begin{align}
d_sd_tx_{st}&=a(d_sd_tw_{st})+b(d_sx_{st},d_tx_{st}),\label{EX}\\
d_sp_{st}&=c(d_sx_{st}),\label{EQ}\\
d_tq_{st}&=e(d_tx_{st})\label{EP}.
\end{align}
Here $w_{st}=(
w^1_{st},\dots,w_{st}^m)$,
with $(w_{st}^i:s,t\ge0)$, $i=1,\dots,m$, independent Brownian sheets,
as above. The unknown processes $(x_{st}:s,t\ge0)$, 
$(p_{st}:s,t\ge0)$ and $(q_{st}:s,t\ge0)$ 
take values in $\R^d$, $\R^n$ and $\R^n$, respectively, 
and are subject to given boundary values
$(x_{s0}:s\ge0)$, $(x_{0t}:t\ge0)$, both assumed to be semimartingales, 
and $(p_{0t}:t\ge0)$, $(q_{s0}:s\ge0)$, both assumed continuous and adapted.
The coefficients $a,b,c,e$ are allowed to have a locally Lipschitz 
dependence on the unknown processes, with the restriction that $b$ depends only on $x$.
Thus, for example, we would write
$a(x_{st},p_{st},q_{st},d_sd_tw_{st})$ and $b(x_{st},d_sx_{st},d_tx_{st})$, 
but have not done so in order to keep the notation compact. 
Moreover, we allow a dependence on the differentials which is a sum
of linear and quadratic terms. Thus, in an expanded notation, we would write
\begin{align*}
d_sd_tx_{st}&=a_1(d_sd_tw_{st})+a_2(d_sd_tw_{st},d_sd_tw_{st})\\
  &\q+b_{11}(d_sx_{st},d_tx_{st})+b_{12}(d_sx_{st},d_tx_{st},d_tx_{st}),\\
  &\q+b_{21}(d_sx_{st},d_sx_{st},d_tx_{st})+b_{22}(d_sx_{st},d_sx_{st},d_tx_{st},d_tx_{st}),\\
d_sp_{st}&=c_1(d_sx_{st})+c_2(d_sx_{st},d_sx_{st}),\\
d_tq_{st}&=e_1(d_tx_{st})+e_2(d_tx_{st},d_tx_{st}),
\end{align*}
where, for $i,j,k=1,2$,
\begin{align*}
a_i:\R^d\times\R^n\times\R^n&\to\R^d\otimes((\R^m)^*)^{\otimes i},\\
b_{jk}:\R^d&\to\R^d\otimes((\R^d)^*)^{\otimes j+k},\\
c_j:\R^d\times\R^n\times\R^n&\to\R^n\otimes((\R^d)^*)^{\otimes j},\\
e_k:\R^d\times\R^n\times\R^n&\to\R^n\otimes((\R^d)^*)^{\otimes k}.
\end{align*}
We may and do assume with loss that $a_2,b_{12}$,$b_{21},b_{22}$,$c_2,e_2$ are 
symmetric in any pair of repeated differential arguments.

By a {\em local solution of} (\ref{EX}--\ref{EP}) {\em with domain $\cD$} 
we mean an adapted initial open set $\cD$,
together with a semimartingale $(x_{st}:(s,t)\in\cD)$,
an $s$-semimartingale $(p_{st}:(s,t)\in\cD)$,
and a $t$-semimartingale $(q_{st}:(s,t)\in\cD)$,
all continuous on $\cD$, such that, for all $(s,t)\in\cD$,
\begin{align*}
x_{st}&=x_{s0}+x_{0t}-x_{00}+\int_0^s\int_0^ta(d_rd_uw_{ru})+\int_0^s\int_0^tb(d_rx_{ru},d_ux_{ru}),\\
p_{st}&=p_{0t}+\int_0^sc(d_rx_{rt}),\\
q_{st}&=q_{s0}+\int_0^te(d_ux_{su}).
\end{align*}
Given such a solution, for each $t\ge0$, we can define 
processes $(u_{st}:(s,t)\in\cD)$ and $(u_{st}^*:(s,t)\in\cD)$, taking values
in $\R^d\times(\R^d)^*$ and $\R^d\times(\R^d)^*\times(\R^d)^*$ respectively, by solving the linear one-parameter
stochastic differential equations
\begin{align}
d_su_{st}&=b_{11}(d_sx_{st},\cdot)u_{st}+b_{12}(d_sx_{st},d_sx_{st},\cdot)u_{st},\label{EU}\\
d_su_{st}^*&=u_{st}^{-1}\{b_{12}(d_sx_{st},u_{st}\cdot,u_{st}\cdot)\notag\\
&\q\q+b_{22}(d_sx_{st},d_sx_{st},u_{st}\cdot,u_{st}\cdot)
-b_{11}(d_sx_{st},b_{12}(d_sx_{st},u_{st}\cdot,u_{st}\cdot))\}.\label{EUS}
\end{align}
Here $u_{st}^{-1}$ denotes the inverse of the linear map $u_{st}$. For fixed $t\ge0$, 
almost surely, $u_{st}$ remains in the set of invertible maps while $(s,t)\in\cD$.
To see this, one can obtain formally a linear equation for the process $(u_{st}^{-1}:(s,t)\in\cD)$,
and then check that its solution is indeed an inverse for $u_{st}$.
Similarly, for each $s\ge0$, we can define
processes $(v_{st}:(s,t)\in\cD)$ and $(v_{st}^*:(s,t)\in\cD)$, taking values
in $\R^d\times(\R^d)^*$ and $\R^d\times(\R^d)^*\times(\R^d)^*$, by solving the analogous equations
\begin{align}
d_tv_{st}&=b_{11}(\cdot,d_tx_{st})v_{st}+b_{21}(\cdot,d_tx_{st},d_tx_{st})v_{st}.\label{EV}\\
d_tv_{st}^*&=v_{st}^{-1}\{b_{21}(v_{st}\cdot,v_{st}\cdot,d_tx_{st})\notag\\
&\q\q+b_{22}(v_{st}\cdot,v_{st}\cdot,d_tx_{st},d_tx_{st})
-b_{11}(b_{21}(v_{st}\cdot,v_{st}\cdot,d_tx_{st}),d_tx_{st})\}.\label{EVS}
\end{align}
We specify initial conditions $u_{00}=v_{00}=I$, so determining
completely $(u_{0s}:s\ge0)$ and $(v_{0t}:t\ge0)$. Then we complete the determination
of the above processes by specifying that $u_{0t}=v_{0t}$, $u_{0t}^*=0$, $v_{s0}=u_{s0}$, and $v_{s0}^*=0$
for all $s,t\ge0$.
Let us say that $(x_{st},p_{st},q_{st}:(s,t)\in\cD)$ is a {\em regular} local solution\footnote{It is not hard 
to see that, for any local solution, the processes just defined have previsible versions,
which are then $s$-semimartingales or $t$-semimartingales, depending on the variable of integration. 
However, we have not determined whether they have a continuous version in general.} 
if there exist continuous $s$-semimartingales $(u_{st}:(s,t)\in\cD)$ and $(u_{st}^*:(s,t)\in\cD)$
satisfying, for each $t\ge0$, the equations (\ref{EU}--\ref{EUS}), and if there exist also 
continuous $t$-semimartingales $(v_{st}:(s,t)\in\cD)$ and $(v_{st}^*:(s,t)\in\cD)$
satisfying, for each $s\ge0$, the equations (\ref{EV}--\ref{EVS}).
A local solution is {\em maximal} if it is not the restriction of any local solution
with larger domain. The notion of a maximal regular local solution is defined analogously.
We assume that the boundary semimartingales 
$(x_{s0}:s\ge0)$, $(x_{0t}:t\ge0)$, $(p_{0t}:t\ge0)$ and $(q_{s0}:s\ge0)$
are {\em regular}\footnote{No connection with the notion of regular local solution is intended.}. 
By this we mean that the Lebesgue--Stieltjes measures defined by their
quadratic variation processes and by the total variation processes of their finite
variation parts are all dominated by $Kds$, or $Kdt$ as appropriate, for some constant $K<\infty$.
We give a result first for the case where $b=0$.

\begin{lemma}\label{BEZ}
Assume that $b=0$. Let $U$ be an open subset of $\R^d\times\R^n\times\R^n$ and let $m:U\to[0,\infty)$ 
be a continuous function with $m(x,p,q)\to\infty$ as $(x,p,q)\to\pd U$.
Assume that, for all $M\ge1$, the coefficients $a,c,e$ are bounded and Lipschitz
on the set $U_M=\{(x,p,q)\in U:m(x,p,q)<M\}$. 
Then, for any set of regular boundary semimartingales
$(x_{s0}:s\ge0)$, $(x_{0t}:t\ge0)$, $(p_{0t}:t\ge0)$ and $(q_{s0}:s\ge0)$,
with $(x_{00},p_{00},q_{00})\in U$,
the equations {\rm (\ref{EX}--\ref{EP})} have a unique maximal local solution $(x_{st},p_{st},q_{st}:(s,t)\in\cD)$
with values in $U$.
Moreover, we have, almost surely\footnote{To clarify, we mean that, for all $(s^*,t^*)\in\z(\cD)$,
the given limit holds whenever $(s,t)\ua(s^*,t^*)$. In particular, in the case where $\cD=(\R^+)^2$,
there are no such points $(s^*,t^*)$ and nothing is claimed.}
$$
\sup_{r\le s,u\le t}m(x_{ru},p_{ru},q_{ru})\to\infty\q\text{as}\q(s,t)\ua\z(\cD).
$$
\end{lemma}
\begin{proof}
In the case where $m$ is bounded (so $U_M=U=\R^d\times\R^n\times\R^n$ for large $M$), the
existence of a (global) solution is proved in \cite[Theorem 3.2.2]{MR1347353}.
The proof is of a standard type, using Picard iteration,
Gronwall's lemma and Kolmogorov's continuity criterion, and gives also the uniqueness
of local solutions on the intersections of their domains.
When $m$ is unbounded, we can find, for each $M\ge1$, bounded Lipschitz coefficients
$a_M,c_M,e_M$ on $\R^d\times\R^n\times\R^n$, which agree with $a,c,e$ on $U_M$.
For each $M_0\ge1$, the corresponding global solutions $(x_{st}^M,p_{st}^M,q_{st}^M: s,t\ge0)$ agree, for all
integers $M\ge M_0$, almost surely, on $\cD_{M_0}$, where
$$
\cD_M=\{(s,t)\in(\R^+)^2:\sup_{r\le s,u\le t}m(x_{ru}^M,p_{ru}^M,q_{ru}^M)\le M\}.
$$
Hence, we obtain a local solution with all the claimed properties by setting $\cD=\cup_M\cD_M$ and
by setting, for all
$M\ge1$, $(x_{st},p_{st},q_{st})=(x_{st}^M,p_{st}^M,q_{st}^M)$ for all $(s,t)\in\cD_M\sm\cD_{M-1}$.
\end{proof}
Our main result deals with the case when $b$ is non-zero.

\begin{theorem}\label{TPS}
Assume that the coefficients $a,b,c,e$ are uniformly bounded and Lipschitz.
Then, for each set of regular semimartingale boundary values
$(x_{s0}:s\ge0)$, $(x_{0t}:t\ge0)$, $(p_{0t}:t\ge0)$, $(q_{s0}:s\ge0)$, 
the system of equations {\rm (\ref{EX}--\ref{EP})}
has a unique maximal regular solution, with domain $\cD$ say. 
As $(s,t)\ua\z(\cD)$, we have
\begin{equation}\label{UVD}
m_{st}=\sup_{s'\le s,t'\le t}|(u_{s't'},u_{s't'}^{-1},v_{s't'},v_{s't'}^{-1})|\to\infty.
\end{equation}
Moreover, if $c$ has Lipschitz first and second derivatives and has no dependence on $q$, 
then $(p_{st}:s,t\in\cD)$ is a semimartingale in $\cD$.
\end{theorem}
\begin{proof}
We consider first the question of existence. We follow, to begin, the strategy used
in the proof of \cite[Theorem 3.2.3]{MR1347353}. Consider the following system
of differential equations, for unknown processes $y_{st},z_{st}$, $x_{st}',u_{st},u_{st}^*,p_{st}$,
$x_{st}'',v_{st},v_{st}^*,q_{st}$, taking values in $\R^d,\R^d$, $\R^d,\R^d\otimes(\R^d)^*,
\R^d\otimes(\R^d)^*\otimes(\R^d)^*,\R^n$,$\R^d,\R^d\otimes(\R^d)^*,
\R^d\otimes(\R^d)^*\otimes(\R^d)^*,\R^n$ respectively:
\begin{align}
d_sd_ty_{st}&=u_{st}^{-1}a(d_sd_tw_{st})
-u^*_{st}(u_{st}^{-1}a(d_sd_tw_{st})\otimes u_{st}^{-1}a(d_sd_tw_{st})),\label{EQY}\\
d_sd_tz_{st}&=v_{st}^{-1}a(d_sd_tw_{st})
-v^*_{st}(v_{st}^{-1}a(d_sd_tw_{st})\otimes v_{st}^{-1}a(d_sd_tw_{st})),\label{EQZ}\\
d_sx_{st}'&=v_{st}(d_sz_{st}+v^*_{st}d_sz_{st}\otimes d_sz_{st}),\label{EQXX}\\
d_su_{st}&=b_{11}(v_{st}(d_sz_{st}+v^*_{st}d_sz_{st}\otimes d_sz_{st}),\cdot)u_{st}
+b_{21}(v_{st}d_sz_{st},v_{st}d_sz_{st},\cdot)u_{st},\label{EQUU}\\
d_su_{st}^*&
=u_{st}^{-1}\{b_{12}(v_{st}(d_sz_{st}+v^*_{st}d_sz_{st}\otimes d_sz_{st}),u_{st}\cdot,u_{st}\cdot)\notag\\
&\q\q+b_{22}(v_{st}d_sz_{st},v_{st}d_sz_{st},u_{st}\cdot,u_{st}\cdot)
     -b_{11}(v_{st}d_sz_{st},b_{12}(v_{st}d_sz_{st},u_{st}\cdot,u_{st}\cdot))\},\label{EQUS}\\
d_sp_{st}&=c(v_{st}(d_sz_{st}+v^*_{st}d_sz_{st}\otimes d_sz_{st})),\label{EQPP}\\
d_tx_{st}''&=u_{st}(d_ty_{st}+u^*_{st}d_ty_{st}\otimes d_ty_{st}),\label{EQHX}\\
d_tv_{st}&=b_{11}(\cdot,u_{st}(d_ty_{st}+u^*_{st}d_ty_{st}\otimes d_ty_{st}))v_{st}
+b_{12}(\cdot,u_{st}d_ty_{st},u_{st}d_ty_{st})v_{st},\label{EQVV}\\
d_tv_{st}^*&
=v_{st}^{-1}\{b_{21}(v_{st}\cdot,v_{st}\cdot,u_{st}(d_ty_{st}+u^*_{st}d_ty_{st}\otimes d_ty_{st}))\notag\\
&\q\q+b_{22}(v_{st}\cdot,v_{st}\cdot,u_{st}d_ty_{st},u_{st}d_ty_{st})
     -b_{11}(b_{21}(v_{st}\cdot,v_{st}\cdot,u_{st}d_ty_{st}),u_{st}d_ty_{st})\},\label{EQVS}\\
d_tq_{st}&=e(u_{st}(d_ty_{st}+u^*_{st}d_ty_{st}\otimes d_ty_{st}))\label{EQQQ}.
\end{align}
We evaluate the coefficients $a$, $b$, $c$ and $e$ here at $(x_{st}',p_{st},q_{st})$ (rather than
at $x_{st}''$).
Note that this system has the same form as the system (\ref{EX}--\ref{EP}) with $b=0$.
We use the boundary conditions given above for $u_{st},p_{st},v_{st},q_{st}$.
Define boundary values for $y_{st}$ and $z_{st}$ by
\begin{equation}\label{BVYZ}
d_sy_{s0}=d_sz_{s0}=v_{s0}^{-1}d_sx_{s0},\q d_ty_{0t}=d_tz_{0t}=u_{0t}^{-1}d_tx_{0t},
\q y_{00}=z_{00}=0.
\end{equation}
Set $u_{0t}^*=v_{s0}^*=0$ and use the given boundary values $(x_{0t}:t\ge0)$
for $x_{st}'$ and $(x_{s0}:s\ge0)$ for $x_{st}''$.
Define, on the set $U$ where $u$ and $v$ are invertible,
$$
m(y,z,x',u,u^*,p,x'',v,v^*,q)=|(u,u^{-1},v,v^{-1})|+|(u^*,v^*)|.
$$
Then the preceding lemma applies, to show that 
(\ref{EQY}--\ref{EQQQ})
has a unique maximal local solution 
with the given boundary values, with domain $\cD$ say, such that 
$u_{st}$ and $v_{st}$ are invertible for all $(s,t)\in\cD$, and such that,
almost surely, as $t\ua\z(\cD)$, either
\begin{equation}
m_{st}=\sup_{s'\le s,t'\le t}|(u_{s't'},u_{s't'}^{-1},v_{s't'},v_{s't'}^{-1})|\ua\infty,\label{MTI}
\end{equation}
or
\begin{equation}
n_{st}=\sup_{s'\le s,t'\le t}|(u_{s't'}^*,v_{s't'}^*)|\ua\infty\label{NTI}.
\end{equation}

Now $v_{st}$ and $v_{st}^*$ are continuous $t$-semimartingales (in $\cD$) and $z_{st}$
is a semimartingale. Moreover $d_ta_{st}d_sd_tz_{st}=0$ for any $t$-semimartingale
$a_{st}$. Hence, by \cite[Theorem 2.3.1]{MR1347353}, $x_{st}'$ is a semimartingale and we may
take the $t$-differential in (\ref{EQXX}) to obtain
\begin{align*}
d_sd_tx_{st}'&=d_tv_{st}(d_sz_{st}+v_{st}^*d_sz_{st}\otimes d_sz_{st})\\
&+v_{st}(d_sd_tz_{st}+d_tv_{st}^*d_sz_{st}\otimes d_sz_{st}+v_{st}^*d_sd_tz_{st}\otimes d_sd_tz_{st})
+d_tv_{st}(d_tv^*_{st}d_sz_{st}\otimes d_sz_{st})\\
&=a(d_sd_tw_{st})+b(d_sx_{st}',d_tx_{st}'').
\end{align*}
Similarly, by taking the $s$-differential in (\ref{EQHX}), we obtain
$$
d_sd_tx_{st}''=a(d_sd_tw_{st})+b(d_sx_{st}',d_tx_{st}'').
$$
We also have $x_{00}'=x_{00}''$ and
$$
d_sx_{s0}'=v_{s0}d_sz_{s0}=d_sx_{s0}'',\q
d_tx_{0t}'=u_{0t}d_ty_{0t}=d_tx_{0t}'',
$$
so $x_{st}'=x_{st}''$ for all $(s,t)\in\cD$, almost surely. Denote the common value of these
processes by $x_{st}$. Then $(x_{st}:(s,t)\in\cD)$ satisfies (\ref{EX}). 
On using (\ref{EQXX}) and (\ref{EQHX})
to substitute\footnote{Such substitutions result in differential formulae corresponding
to valid identities between processes. This is because the two-parameter stochastic differential calculus
is associative, as mentioned above, and as discussed in \cite[pp. 290--291]{MR1347353}.}  for $d_sz_{st}$ and $d_ty_{st}$ in (\ref{EQUU}, \ref{EQPP}, \ref{EQVV}, \ref{EQQQ}),
we see also that $p_{st}$, $q_{st}$, $u_{st}$, $u_{st}^*$, $v_{st}$, $v_{st}^*$ 
satisfy (\ref{EQ}--\ref{EVS})
respectively. 
Hence $(x_{st},p_{st},q_{st}:(s,t)\in\cD)$ is a regular local solution to (\ref{EX}--\ref{EP}),
which is moreover maximal by virtue of (\ref{MTI}--\ref{NTI}).

We turn to the question of uniqueness. 
Suppose that $(\tilde x_{st},\tilde p_{st},\tilde q_{st}:(s,t)\in\tilde \cD)$ is any 
regular local solution to (\ref{EX}--\ref{EP}).
Write $(\tilde u_{st},\tilde u^*_{st},\tilde v_{st},\tilde v^*_{st}:(s,t)\in\tilde\cD)$ for the associated 
processes, satisfying (\ref{EU}--\ref{EVS}).
Define semimartingales $(\tilde y_{st}:(s,t)\in\tilde \cD)$ and $(\tilde z_{st}:(s,t)\in\tilde \cD)$
by 
\begin{align}\label{YTST}
d_sd_t\tilde y_{st}&=\tilde u_{st}^{-1}a(d_sd_tw_{st})
-\tilde u^*_{st}(\tilde u_{st}^{-1}a(d_sd_tw_{st})\otimes \tilde u_{st}^{-1}a(d_sd_tw_{st})),\\
\label{ZTST}
d_sd_t\tilde z_{st}&=\tilde v_{st}^{-1}a(d_sd_tw_{st})
-\tilde v^*_{st}(\tilde v_{st}^{-1}a(d_sd_tw_{st})\otimes\tilde v_{st}^{-1}a(d_sd_tw_{st})),
\end{align}
with boundary values (\ref{BVYZ}).
The following equations may be verified by checking that the initial values and differentials 
of left and right hand sides agree
\begin{equation} \label{DSX}
d_s\tilde x_{st}=\tilde v_{st}(d_s\tilde z_{st}+\tilde v^*_{st}d_s\tilde z_{st}\otimes d_s\tilde z_{st}),\q
d_t\tilde x_{st}=\tilde u_{st}(d_t\tilde y_{st}+\tilde u^*_{st}d_t\tilde y_{st}\otimes d_t\tilde y_{st}).
\end{equation}
Then, using these equations to substitute for $d_s\tilde x_{st}$ and $d_t\tilde x_{st}$ in 
(\ref{EQ}--\ref{EVS}), we see that \linebreak
 $(\tilde y_{st},\tilde z_{st},\tilde x_{st},\tilde u_{st},\tilde u^*_{st},\tilde p_{st},\tilde x_{st},\tilde v_{st},\tilde v_{st}^*,\tilde q_{st}:(s,t)\in\tilde \cD)$
is a local solution to (\ref{EQY}--\ref{EQQQ}).
By local uniqueness for this system, $\tilde\cD\sse\cD$ and 
$(\tilde x_{st},\tilde p_{st},\tilde q_{st})=(x_{st},p_{st},q_{st})$
for all $(s,t)\in\tilde\cD$, almost surely. Thus $(x_{st},p_{st},q_{st}:(s,t)\in\cD)$ is the unique 
maximal regular local solution to (\ref{EX}--\ref{EP}).

Our next goal is to obtain $\a$th-moment and $L^\a$-H\" older 
estimates on the process \linebreak $(x_{st},p_{st},q_{st},u_{st},u_{st}^*,v_{st},v_{st}^*:(s,t)\in\cD)$, for $\a\in[2,\infty)$.
Write $K$ for a uniform bound on $a,b,c,e$ which is also a Lipschitz constant for $b$.
Fix $M,N,T\ge1$ and set
\begin{align*}
\cD_M&=\{(s,t)\in\cD:s,t\le T\text{ and }m_{st}\le M\},\\
\cD_{M,N}&=\{(s,t)\in\cD:s,t\le T,m_{st}\le M\text{ and }n_{st}\le N\}.
\end{align*}
Fix $\a$ and define
$$
g(s,t)=\sup_{s'\le s,t'\le t}\E(|(u_{s't'}^*,v_{s't'}^*)|^\a1_{\{(s',t')\in\cD_{M,N}\}}).
$$
Let $(a_s:s\ge0)$ be a locally bounded, $(\cF_{s\infty})_{s\ge0}$-previsible process.
The following identities follow from equations (\ref{ZTST}) and (\ref{DSX}): 
for $(s,t)\in\cD$, respectively in $\R^d$ and $\R^d\otimes\R^d$,
\begin{equation}\label{ADX}
\int_0^sa_rd_rx_{rt}
=\int_0^sa_rd_rx_{r0}
+\int_0^s\int_0^ta_rv_{rt}\left\{v_{ru}^{-1}a(d_rd_uw_{ru})
+(v_{rt}^*-v_{ru}^*)(v_{ru}^{-1}a(d_rd_uw_{ru}))^{\otimes 2}\right\}
\end{equation}
and
\begin{equation}\label{ADXX}
\int_0^sa_rd_rx_{rt}\otimes d_rx_{rt}
=\int_0^sa_rd_rx_{r0}\otimes d_rx_{r0}
+\int_0^s\int_0^ta_r(v_{rt}v_{ru}^{-1}a(d_rd_uw_{ru}))^{\otimes 2}.
\end{equation}
Hence, using the estimate (\ref{BDG}),
we obtain a constant $C=C(\a,K,M,T)<\infty$ such that, for all $s,t\ge0$,
\begin{align}\label{BDGX}
&\E\left(\left|\int_0^s
a_rd_rx_{rt}\right|^\a1_{\{(s,t)\in\cD_{M,N}\}}\right)\notag\\
&\q\q\le C\E\left(\left|\left(\int_0^s
a_r^2dr\right)^{1/2}
+\int_0^s\int_0^t
|a_r|(|v_{rt}^*|+|v_{ru}^*|)drdu\right|^{\a}1_{\{(s,t)\in\cD_{M,N}\}}\right)
\end{align}
and
\begin{equation}\label{BDGXX}
\E\left(\left|\int_0^s
a_rd_rx_{rt}\otimes d_rx_{rt}\right|^\a1_{\{(s,t)\in\cD_{M,N}\}}\right)
\le C\E\left(\left|\int_0^s
|a_r|dr\right|^\a1_{\{(s,t)\in\cD_{M,N}\}}\right).
\end{equation}
Here and below, we suppress any dependence of constants on the dimensions
$d,n,m$. If we allow $C$ to depend also on $N$, then (\ref{BDGX}) may be
simplified to
\begin{equation}\label{BDGXL}
\E\left(\left|\int_0^s
a_rd_rx_{rt}\right|^\a1_{\{(s,t)\in\cD_{M,N}\}}\right)
\le C\E\left(\left|\int_0^s a_r^2dr \right|^{\a/2}1_{\{(s,t)\in\cD_{M,N}\}}\right)
\end{equation}
We use these estimates, along with analogous estimates for integrals $d_tx_{st}$,
in the equations (\ref{EUS}) and (\ref{EVS}), to arrive at the inequality
$$
g(s,t)\le C\left(1+\int_0^sg(s',t)ds'+\int_0^tg(s,t')dt'\right),
$$
for a constant $C=C(\a,K,M,T)<\infty$. Since $N<\infty$, we know that $g(s,t)<\infty$
for all $s,t$, so this inequality implies that $g(s,t)\le C$ for another constant 
$C<\infty$ of the same dependence.
Similar arguments yield
a further constant $C<\infty$ of the same dependence such that, for all $s,s'\ge0$ and all $t,t'\ge0$,
\begin{equation}\label{HSU}
\E(|(x_{st},u_{st},u^*_{st},p_{st})-(x_{s't},u_{s't},u^*_{s't},p_{s't})|^\a1_{\{(s,t),(s',t)\in\cD_{M,N}\}})
\le C|s-s'|^{\a/2}
\end{equation}
and
\begin{equation}\label{HSV}
\E(|(x_{st},v_{st},v^*_{st},q_{st})-(x_{st'},v_{st'},v^*_{st'},q_{st'})|^\a1_{\{(s,t),(s,t')\in\cD_{M,N}\}})
\le C|t-t'|^{\a/2}.
\end{equation}
Here, we have used Cauchy--Schwarz to obtain in an intermediate step
$$
\int_s^{s'}\int_0^t|v_{ru}^*|drdu\le|s-s'|^{1/2}\left(\int_s^{s'}\int_0^t|v_{ru}^*|^2drdu\right)^{1/2}.
$$
On going back to (\ref{ADX}) and (\ref{ADXX}) with these H\" older estimates, we obtain, using
(\ref{BDG}) again, a constant $C<\infty$ of the same dependence such that
\begin{equation}\label{TTPX}
\E\left(\left|\int_0^s a_r(d_rx_{rt}-d_rx_{rt'})\right|^\a1_{\{(s,t),(s,t')\in\cD_{M,N}\}}\right)
\le C|t-t'|^{\a/2}\left(\E\left|\int_0^sa_r^2ds\right|^\a\right)^{1/2}
\end{equation}
and
\begin{align}
&\E\left(\left|\int_0^s a_rd_rx_{rt}\otimes(d_rx_{rt}-d_rx_{rt'})\right|^\a1_{\{(s,t),(s,t')\in\cD_{M,N}\}}\right)\notag\\
&\q\q\q\q\q\q\q\q\q\q\q\q\q\q\q\q\q\q\q\q\le C|t-t'|^{\a/2}\left(\E\left|\int_0^sa_r^2ds\right|^\a\right)^{1/2}.\label{TTPXX}
\end{align}
Now
\begin{align*}
d_s(u_{st}^{-1}u_{st'})&=u_{st}^{-1}\{b(x_{st'},d_sx_{st'},\cdot)-b(x_{st},d_sx_{st},\cdot)\}u_{st'}\\
&\q\q
-u_{st}^{-1}b_{11}(x_{st},d_sx_{st},\cdot)
\{b_{11}(x_{st'},d_sx_{st'},\cdot)-b_{11}(x_{st},d_sx_{st},\cdot)\}u_{st'}.
\end{align*}
We have made explicit the dependence of $b$ and $b_{11}$ on $x_{st}$ or $x_{st'}$.
We use the estimates (\ref{BDGX}), (\ref{BDGXX}), (\ref{HSV}--\ref{TTPXX}) to find a
constant $C=C(\a,K,M,T)<\infty$ such that
\begin{equation}\label{HTU}
\E(|u_{st}-u_{st'}|^\a1_{\{(s,t),(s,t')\in\cD_{M,N}\}})\le C|t-t'|^{\a/2}.
\end{equation}
Moreover, the same estimates, applied to the difference of (\ref{EUS}) at $t$ and at $t'$,
show that $C$ may be chosen such that
\begin{equation}\label{HTUS}
\E(|u_{st}^*-u_{st'}^*|^\a1_{\{(s,t),(s,t')\in\cD_{M,N}\}})\le C|t-t'|^{\a/2}.
\end{equation}
Since $C$ does not depend on $N$, by monotone convergence, 
we can replace $\cD_{M,N}$
by $\cD_M$ in these estimates
By symmetry, there are analogous estimates for $v_{st}$ and $v_{st}^*$.
Hence, using \cite[Theorem 3.2.1]{MR1347353}, almost surely, for all $M\ge1$,
$n_{st}$ remains bounded on $\cD_M$. Thus (\ref{NTI}) implies (\ref{MTI})
so, in any case, (\ref{UVD}) holds.

It remains to consider the case where $c$ has Lipschitz first and second derivatives and has
no dependence on $q$, and to show then that $(p_{st}:(s,t)\in\cD)$ is a semimartingale.
For ease of writing, we shall assume that $c$ has no dependence on $x$ either. This
is done without loss of generality, by the device of adding to our system the
equation $d_sx_{st}=d_sx_{st}$, thus making $x_{st}$ a component of $p_{st}$.

We seek to find a solution in a smaller class of processes, in which $p_{st}$ is a semimartingale. 
Recall that 
\begin{equation}\label{NPE}
d_sp_{st}=c(d_sx_{st})=c_1(p_{st})(d_sx_{st})+c_2(p_{st})(d_sx_{st},d_sx_{st}).
\end{equation}
By It\^ o's formula, if $p_{st}$ is a semimartingale, then
\begin{align*}
d_sd_t p_{st}
&=c'(d_tp_{st},d_sx_{st})+\tfrac12c''(d_tp_{st},d_tp_{st},d_sx_{st})+c(d_sd_tx_{st})+c'(d_tp_{st},d_sd_tx_{st})\\
&\q\q\q\q+2c_2(d_sx_{st},d_sd_tx_{st})+2c_2'(d_tp_{st},d_sx_{st},d_sd_tx_{st})\\
&=c'(d_tp_{st},d_sx_{st})+\tfrac12c''(d_tp_{st},d_tp_{st},d_sx_{st})+c(a(d_sd_tw_{st}))+c(b(d_sx_{st},d_tx_{st}))\\
&\q\q\q\q
+c'(d_tp_{st},b(d_sx_{st},d_tx_{st}))
+2c_2(d_sx_{st},b(d_sx_{st},d_tx_{st}))\\
&\q\q\q\q
+2c_2'(d_tp_{st},d_sx_{st},b(d_sx_{st},d_tx_{st})).
\end{align*}
Here we are writing $c',c''$ for the derivatives with respect to $p$.
We set $\tilde d=d+n$ and combine this equation with the 
equation (\ref{EX}) to obtain a
two-parameter equation for the $\R^{\tilde d}$-valued process 
$\tilde x_{st}=
\begin{pmatrix}
x_{st}\\p_{st}
\end{pmatrix}$, which
we can write in the form
\begin{align}
d_sd_t\tilde x_{st}&=\tilde a(d_sd_tw_{st})+\tilde b(d_s\tilde x_{st},d_t\tilde x_{st}).\label{EQTX}
\end{align}
(The $\sim$ notation in this paragraph has nothing to do with that used
in the paragraph on uniqueness above.)
We impose regular semimartingale initial values 
$\tilde x_{s0}=
\begin{pmatrix}
x_{s0}\\p_{s0}
\end{pmatrix}$ and $\tilde x_{0t}=
\begin{pmatrix}
x_{0t}\\p_{0t}
\end{pmatrix}$, where $(p_{s0}:s\ge0)$ is obtained by solving the 
one-parameter
equation (\ref{NPE}) along $x_{s0}$. 
Introduce the two companion equations for $\tilde d\times\tilde d$ matrix-valued
processes $\tilde u_{st}$ and $\tilde v_{st}$
\begin{align}
d_s\tilde u_{st}&=\tilde b_{11}(d_s\tilde x_{st},\cdot)\tilde u_{st}+\tilde b_{12}(d_s\tilde x_{st},d_s\tilde x_{st},\cdot)\tilde u_{st},\label{EQTU}\\
d_t\tilde v_{st}&=\tilde b_{11}(\cdot,d_t\tilde x_{st})\tilde v_{st}+\tilde b_{21}(\cdot,d_t\tilde x_{st},d_t\tilde x_{st})\tilde v_{st}.\label{EQTV}
\end{align}
Impose boundary conditions for $\tilde u_{st}$ 
and $\tilde v_{st}$ analogous to those for $u_{st}$ and $v_{st}$.
Write (\ref{EP}) in the form
\begin{equation}\label{EQQT}
d_t\tilde q_{st}=\tilde e(d_t\tilde x_{st}).
\end{equation}
By assumption, there exists a $K'<\infty$ which is both a uniform bound for $a,b,c,e$ and is also 
a Lipschitz constant for $b,c,c',c''$.
We can then find a uniform bound $\tilde K<\infty$ on $\tilde a,\tilde b, \tilde e$,
 which is also a Lipschitz constant for $\tilde b$, and which depends only on $K'$. 
The above argument shows that the system of equations (\ref{EQTX}--\ref{EQQT})
has a unique maximal regular solution $(\tilde x_{st},\tilde q_{st},\tilde u_{st},\tilde v_{st}:(s,t)\in\tilde\cD)$, 
with the property that,
as $(s,t)\ua\zeta(\tilde\cD)$, almost surely,
$$
\tilde m_{st} :=\sup_{r\leq s,u\leq t} |(\tilde u_{ru},\tilde u_{ru}^{-1},\tilde v_{ru},\tilde v_{ru}^{-1})|\ua\infty.
$$
Write
$$
\tilde x_{st}=
\begin{pmatrix}
x_{st}^1\\x_{st}^2
\end{pmatrix},\q
\tilde u_{st}=
\begin{pmatrix}
u_{st}^{11}&u^{12}_{st}\\ u^{21}_{st}&u^{22}_{st}
\end{pmatrix},\q
\tilde v_{st}=
\begin{pmatrix}
v_{st}^{11}&v^{12}_{st}\\ v^{21}_{st}&v^{22}_{st}
\end{pmatrix},
$$
and use analogous block notation for the tensors $\tilde u_{st}^*$ and $\tilde v_{st}^*$.
Note that
$$
\tilde b(d_s\tilde x_{st},\cdot)
=\begin{pmatrix}
b(d_sx_{st}^1,\cdot)&0\\ f(d_sx_{st}^1)&c'(\cdot,d_sx_{st}^1)
\end{pmatrix},\q
\tilde b(\cdot,d_t\tilde x_{st})
=\begin{pmatrix}
b(\cdot,d_tx_{st}^1)&0\\ g(d_t\tilde x_{st})&0
\end{pmatrix},
$$
where
\begin{align*}
f(d_sx_{st}^1)&=c(b(d_sx_{st}^1,\cdot))+2c_2(d_sx_{st}^1,b(d_sx_{st}^1,\cdot)),\\
g(d_t\tilde x_{st})&=c'(d_tx_{st}^2,\cdot)+\tfrac12c''(d_tx_{st}^2,d_tx_{st}^2,\cdot)
+c(b(\cdot,d_tx_{st}^1))+c'(d_tx_{st}^2,b(\cdot,d_tx_{st}^1)).
\end{align*}
Here, we have written $b(d_sx_{st},\cdot)$ as a short form of $b_{11}(d_sx_{st},\cdot)+b_{12}(d_sx_{st},d_sx_{st},\cdot)$, and
analogously for $b(\cdot,d_tx_{st})$ and $\tilde b(d_s\tilde x_{st},\cdot)$. 
On multiplying out in blocks, we see that the process
$(x_{st}^1,x_{st}^2,\tilde q_{st},u_{st}^{11},(u_{st}^*)^{111},v_{st}^{11},(v_{st}^*)^{111}:(s,t)\in\tilde\cD)$ 
satisfies equations (\ref{EX}--\ref{EVS}).
Hence, we must have $\tilde\cD\sse\cD$ and 
$(x_{st}^1,x_{st}^2,\tilde q_{st},u_{st}^{11},v_{st}^{11})=(x_{st},p_{st},q_{st},u_{st},v_{st})$ for all $(s,t)\in\tilde\cD$. 
In particular, $(p_{st}:(s,t)\in\tilde\cD)$ is a semimartingale.

It remains to show that $\tilde\cD=\cD$, which we can do by showing that,
almost surely, $\tilde m_{st}$ remains bounded on $\tilde\cD_{M,N}=\tilde\cD\cap\cD_{M,N}$, for all $M,N\ge1$.
We first obtain a H\" older estimate in $t$ for $p_{st}$.
We have 
$$
d_s(p_{st}-p_{st'})=c(p_{st},d_sx_{st})-c(p_{st'},d_sx_{st'}),
$$
where we have now made the dependence of $c$ on $p$ explicit.
Set
$$
f(s)=\E\left(|p_{st}-p_{st'}|^\a1_{\{(s,t),(s,t')\in\tilde\cD_{M,N}\}}\right).
$$
We use the estimates (\ref{BDGXX}) and (\ref{BDGXL}) to obtain
a constant $C=C(\a,K',M,N,T)<\infty$ such that
$$
f(s)\le C\left(|t-t'|^{\a/2}+\int_0^sf(r)dr\right).
$$
This implies that $f(s)\le C|t-t'|^{\a/2}$ for all $s\ge0$ for a constant $C<\infty$
of the same dependence.
We now know that, for such a constant $C<\infty$, we have
\begin{equation}\label{PSTH}
\E\left(\left|p_{s't'}-p_{st}\right|^\a
1_{\{(s,t),(s',t')\in\tilde\cD_{M,N}\}}\right)
\le C(|s-s'|^{\a/2}+|t-t'|^{\a/2}).
\end{equation}
We turn to $\tilde u_{st}$ and $\tilde v_{st}$.
The following equations hold
$$
d_su_{st}^{12}=b(d_sx_{st},\cdot)u_{st}^{12},\q
d_tv_{st}^{12}=b(\cdot,d_tx_{st})v_{st}^{12},\q
d_tv_{st}^{22}=g(d_t\tilde x_{st})v_{st}^{12}.
$$
By uniqueness of solutions, we obtain 
$u^{12}_{st}=u_{st}u_{0t}^{-1}u_{0t}^{12}$
so, in particular, $u^{12}_{s0}=0$.
Similarly, $v^{12}_{st}=v_{st}v_{s0}^{-1}v_{s0}^{12}$,
so $v^{12}_{0t}=0$.
Since $\tilde u_{0t}=\tilde v_{0t}$ and $\tilde u_{s0}=\tilde v_{s0}$,
we deduce that $u_{st}^{12}=v_{st}^{12}=0$.
Then $d_tv^{22}_{st}=0$, so $v^{22}_{st}=v^{22}_{s0}=u^{22}_{s0}$.
We also have the equations
$$
d_su^{21}_{st}=f(d_sx_{st})u_{st}+c'(.,d_sx_{st})u^{21}_{st},\q
d_su^{22}_{st}=c'(.,d_sx_{st})u^{22}_{st},\q
d_tv^{21}_{st}=g(d_t\tilde x_{st})v_{st}
$$
and we note that
$$
\tilde u^{-1}_{st}=
\begin{pmatrix}
u^{-1}_{st}&0\\ -(u^{22}_{st})^{-1}u^{21}_{st}u_{st}^{-1}&(u^{22}_{st})^{-1}
\end{pmatrix},\q
\tilde v^{-1}_{st}=
\begin{pmatrix}
v^{-1}_{st}&0\\ -(v^{22}_{st})^{-1}v^{21}_{st}v_{st}^{-1}&(v^{22}_{st})^{-1}
\end{pmatrix},
$$
and 
$$
d_s(u^{22}_{st})^{-1}=-(u^{22}_{st})^{-1}c'(.,d_sx_{st})+(u^{22}_{st})^{-1}c'(.,d_sx_{st})c'(.,d_sx_{st}).
$$
We use the inequalities (\ref{BDGXX}), (\ref{BDGXL}) and (\ref{PSTH}), and an
easy variation of the argument leading to (\ref{HSU}) and (\ref{HTU}) to obtain
a constant $C=C(\a,K',M,N,T)<\infty$ such that
\begin{equation}\label{TUH}
\E\left(\left|(\tilde u_{s't'},\tilde u_{s't'}^{-1})-(\tilde u_{st},\tilde u_{st}^{-1})\right|^\a
1_{\{(s,t),(s',t')\in\tilde\cD_{M,N}\}}\right)
\le C(|s-s'|^{\a/2}+|t-t'|^{\a/2}).
\end{equation}
Then, using \cite[Theorem 3.2.1]{MR1347353} as above, we can conclude that,
almost surely, $(\tilde u_{st},\tilde u_{st}^{-1})$ remains bounded on
$\tilde\cD_{M,N}$.
It remains to show that the same is true for 
$(\tilde v_{st},\tilde v_{st}^{-1})$ and, given the relations already noted,
it will suffice to show this for $v_{st}^{21}$. We have
\begin{align*}
d_s\tilde u_{st}^*&=\tilde u_{st}^{-1}\{\tilde b_{12}(d_s\tilde x_{st},\tilde u_{st}\cdot,\tilde u_{st}\cdot)\\
&\q\q+\tilde b_{22}(d_s\tilde x_{st},d_s\tilde x_{st},\tilde u_{st}\cdot,\tilde u_{st}\cdot)
-\tilde b_{11}(d_s\tilde x_{st},\tilde b_{12}(d_s\tilde x_{st},\tilde u_{st}\cdot,\tilde u_{st}\cdot))\}\\
&=h(x_{st},p_{st},\tilde u_{st},\tilde u_{st}^{-1},d_sx_{st}),
\end{align*}
where $h$ is defined by the final equality and where we have used (\ref{EQ}) to write $d_s\tilde x_{st}$ in terms
of $d_sx_{st}$.
A variation of the argument used for $\tilde u_{st}$ shows that, almost surely, $\tilde u_{st}^*$
remains bounded on $\tilde\cD_{M,N}$.
Then, we can use the $\sim$ and $t$-analogue of equations (\ref{ADX}) and (\ref{ADXX})
to express $v_{st}^{21}$ as a sum of integrals with respect to $(x_{0t},p_{0t}:t\ge0)$
and $(w_{st}:s,t\ge0)$. This leads, as above, to $L^\a$-H\" older estimates which allow us
to conclude that, almost surely, $v_{st}^{21}$ remains
bounded on $\tilde\cD_{M,N}$, as required.
\end{proof}
\def\backup{\bg
Write 
$$
g(d_t\tilde x_{st})=g_{st}d_t\tilde x_{st}+g^{(2)}_{st}d_t\tilde x_{st}\otimes d_t\tilde x_{st},
$$
where $g_{st}=\g(x_{st},p_{st})$ and $g^{(2)}_{st}=\g^{(2)}(x_{st},p_{st})$, for some 
functions $\g,\g^{(2)}$, which can be expressed in terms of $a,b,c$. Then 
\begin{align*}
v_{st}^{21}&=v_{s0}^{21}
+\int_0^tg_{su}d_u\tilde x_{0u}v_{su}
+\int_0^tg_{su}^{(2)}(d_u\tilde x_{0u}\otimes d_u\tilde x_{0u})v_{su}\\
&\q\q
\int_0^s\int_0^tg_{su}\tilde u_{su}\{\tilde u_{ru}^{-1}\tilde a(d_rd_uw_{ru})
+(\tilde u_{su}^*-\tilde u_{ru}^*)(\tilde u_{ru}^{-1}\tilde a(d_rd_uw_{ru}))^{\otimes2}\}v_{su}\\
&\q\q
\int_0^s\int_0^tg_{su}^{(2)}(\tilde u_{su}\tilde u_{ru}^{-1}\tilde a(d_rd_uw_{ru}))^{\otimes2}v_{su}.
\end{align*}
We already have H\" older estimates on all the factors in the integrands, so this expression
allows us to deduce the claimed estimates on $v_{st}^{21}$.
\eg}

\section{Derivation of the formula}\label{C}
Let $(w_{st}:s,t\ge0)$ be an $\R^m$-valued Brownian sheet and let
$(z_{s0}:s\ge0)$ be an independent $\R^m$-valued Brownian motion.
Thus $w_{st}=(w^1_{st},\dots,w^m_{st})$ and $z_{s0}=(z^1_{s0},\dots,z^m_{s0})$,
and each component process is an independent scalar Brownian sheet, or Brownian
motion, respectively. The two-parameter hyperbolic stochastic differential
equation
\begin{equation}\label{OUSDE}
d_sd_tz_{st}=d_sd_tw_{st}-\tfrac12d_sz_{st}dt, \q s,t\ge0,
\end{equation}
with given boundary values $(z_{s0}:s\ge0)$ and $z_{0t}=0$, for $t\ge0$, 
has a unique
solution $(z_{st}:s,t\ge0)$. Set $z_t=(z_{st}:s\ge0)$, then $(z_t)_{t\ge0}$ 
is a realization of the Ornstein-Uhlenbeck process on the 
$m$-dimensional Wiener space. See \cite{MR536013} or \cite{MR1347353}.
The Stratonovich form of (\ref{OUSDE}) is given by
$$
\ds\dt z_{st}=\ds\dt w_{st}-\tfrac12\ds z_{st}\pd t, \q s,t\ge0.
$$
Fix $x\in\R^d$ and consider for each $t\ge0$ the Stratonovich
stochastic differential equation 
$$
\ds x_{st}=X_i(x_{st})\ds z^i_{st}+X_0(x_{st})\pd s, \q s\ge0,
$$
with initial value $x_{0t}=x$. This can be written in It\^ o form as
\begin{equation}\label{XSDE}
d_s x_{st}=X_i(x_{st})d_s z^i_{st}+\tilde X_0(x_{st})ds, \q s\ge0,
\end{equation}
where $\tilde X_0=X_0+\frac12\sum_{i=1}^d\nabla X_i.X_i$.
Consider also, for each $t\ge0$, the stochastic differential equation
$$
\ds U_{st}=\nabla X_i(x_{st})U_{st}\ds z^i_{st}+\nabla X_0(x_{st})U_{st}\pd s, \q s\ge0,
$$
with initial value $U_{0t}=I$, and its It\^ o form 
\begin{equation}\label{USDE}
d_s U_{st}=\nabla X_i(x_{st})U_{st}d_s z^i_{st}+\nabla\tilde X_0(x_{st})U_{st}ds, \q s\ge0.
\end{equation}

\begin{proposition}
There exist (two-parameter) semimartingales $(z_{st}:s,t\ge0)$, $(x_{st}:s,t\ge0)$
and $(U_{st}:s,t\ge0)$ such that $(z_{st}:s,t\ge0)$ satisfies (\ref{OUSDE}) and, for all
$t\ge0$, $(x_{st}:s\ge0)$ and $(U_{st}:s\ge0)$ satisfy (\ref{XSDE}) and (\ref{USDE}),
with the boundary conditions given above. Moreover, almost surely, $U_{st}$ is invertible
for all $s,t\ge0$.
\end{proposition}
\begin{proof}
We seek to apply Theorem \ref{TPS}. There are three minor obstacles: firstly to deal with the
$ds$ and $dt$ differentials appearing in the equations, secondly, to show that the domain of the
solutions is the whole of $(\R^+)^2$ and, thirdly, to deal with the fact that the coefficients
in (\ref{USDE}) do not have the required boundedness of derivatives.

Let us introduce a further equation
$$
d_sd_tz^0_{st}=0,
$$
with boundary conditions $z^0_{s0}=s$ and $z_{0t}=t$ for all $s,t\ge0$.
We then replace $dt$ and $ds$ in (\ref{OUSDE}) and (\ref{XSDE}), respectively, by $d_tz^0_{st}$
and $d_sz_{st}^0$.
When we obtain a solution, it will follow that $z^0_{st}=s+t$, so $d_tz^0_{st}=dt$
and $d_sz_{st}^0=ds$, as required.

In order to show that $\cD=(\R^+)^2$, it will suffice to show that the companion processes
$u_{st}$ and $v_{st}$ associated with the equations
$$
d_sd_tz^0_{st}=0,\q d_sd_tz_{st}=d_sd_tw_{st}-\frac12d_sz_{st}d_tz^0_{st},
$$
according to equations (\ref{EU}) and (\ref{EV}),
along with their inverses, remain bounded on compacts in $s$ and $t$. 
We leave this to the reader.

Finally, choose for each $M\in\N$ a smooth and compactly supported function $\psi_M$
on $\R^d\otimes(\R^d)^*$, such that $\psi_M(U)=U$ whenever $|U|\le M$. We can apply Theorem \ref{TPS}
to the system (\ref{OUSDE}), (\ref{XSDE}), together with the modified equation
\begin{equation*}\label{wSDE}
d_s U_{st}^M=\nabla X_i(x_{st})\psi_M(U_{st}^M)d_s z^i_{st}+\nabla\tilde X_0(x_{st})\psi_M(U_{st}^M)ds.
\end{equation*}
Define
$$
\cD_M=\{(s,t):|U_{s't'}^M|\le M\text{ for all }s'\le s,t'\le t\}.
$$
By local uniqueness, we can define consistently $U$ on $\cD=\cup_M\cD_M$ by $U_{st}=U^M_{st}$
for $(s,t)\in\cD_M$. By some straightforward estimation using the one-parameter
equations (\ref{USDE}), we obtain, for all $T<\infty$ and all $p\in[1,\infty)$, a constant
$C<\infty$ such that
$$
\sup_{s,s',t,t'\le T}\E(|U_{st}-U_{s't'}|^p1_{\{(s,t),(s't')\in\cD\}})\le C(|s-s'|^{p/2}+|t-t'|^{p/2}).
$$
Then, by  \cite[Theorem 3.2.1]{MR1347353}, almost surely, $U$ is bounded uniformly
on $\cD\cap[0,T]^2$. Hence $\cD=(\R^+)^2$, and we have obtained the desired semimartingale $U$.
The invertibility of $U$ can be proved by applying the same argument to the usual equation
for the inverse.
\end{proof}

By the Stratonovich chain rule,
$$ 
\ds\dt x_{st}=\nabla X_i(x_{st})\ds z^i_{st}\dt x_{st}
+\nabla X_0(x_{st})\pd s\dt x_{st}
+X_i(x_{st})\ds\dt z^i_{st}.
$$
Now
\begin{align*}
\ds\dt U_{st}&=
\nabla X_i(x_{st})\ds z^i_{st}\dt U_{st}
+\nabla X_0(x_{st})\pd s\dt U_{st}\\
&\q\q+(\nabla^2 X_i(x_{st})\dt x_{st})U_{st}\ds z^i_{st}
+(\nabla^2 X_0(x_{st})\dt x_{st})U_{st}\pd s
+\nabla X_i(x_{st})U_{st}\ds\dt z^i_{st},
\end{align*}
so 
$$
\dt U_{st}\ds\dt z_{st}^i=\tfrac12\ds\dt U_{st}\ds\dt w_{st}^i
=\tfrac12\nabla X_i(x_{st})U_{st}\pd s\pd t
$$
and
$$
\ds(U^{-1}_{st}\dt U_{st})
=U^{-1}_{st}\left\{
\nabla^2X_i(x_{st})\ds z^i_{st}\dt x_{st}
+\nabla^2X_0(x_{st})\pd s\dt x_{st}
+\nabla X_i(x_{st})\ds\dt z^i_{st}\right\}U_{st}.
$$
Define also a two-parameter, $\R^d$-valued, semimartingale $(y_{st}:s,t\ge0)$ by
$$
\dt y_{st}=U^{-1}_{st}\dt x_{st}, \q y_{s0}=0.
$$
Then
$$
\ds\dt y_{st}=U^{-1}_{st}X_i(x_{st})\ds\dt z^i_{st}.
$$
Note that
$$
\dt y_{st}\ds\dt z_{st}^i=\dt y_{st}\ds\dt w_{st}^i
=\tfrac12\ds\dt y_{st}\ds\dt w_{st}^i=\tfrac12U^{-1}_{st}X_i(x_{st})\pd s\pd t.
$$
So
$$
\ds(\dt y_{st}\otimes\dt y_{st})
=\ds\dt y_{st}\otimes\dt y_{st}+\dt y_{st}\otimes\ds\dt y_{st}
=U^{-1}_{st}X_i(x_{st})\otimes U^{-1}_{st}X_i(x_{st})\pd s\pd t.
$$
Note also that
$$
\ds(U^{-1}_{st}X_i(x_{st}))=U_{st}^{-1}[X_i,X_j](x_{st})\ds z^j_{st}
+U_{st}^{-1}[X_i,X_0](x_{st})\pd s.
$$
So
$$
\ds(U^{-1}_{st}X_i(x_{st}))\ds\dt z^i_{st}
=U_{st}^{-1}[X_i,X_j](x_{st})\ds z^j_{st}(\ds\dt w_{st}^i-\tfrac12\ds z_{st}^i\pd t)=0.
$$ 
Moreover
$$
\dt(U^{-1}_{st}X_i(x_{st}))d_s\dt z^i_{st}
=\dt(U^{-1}_{st}X_i(x_{st}))d_s\dt w^i_{st}=0.
$$
Hence, we have
$$
d_sd_t y_{st}=U^{-1}_{st}X_i(x_{st})d_sd_t z^i_{st}=U^{-1}_{st}X_i(x_{st})(\ds\dt w_{st}^i-
\tfrac12\ds z_{st}^i\pd t).
$$
We compute
\begin{align*}
&\ds(U^{-1}_{st}\dt U_{st}\dt y_{st})\\
&\q\q=U^{-1}_{st}\left\{
\nabla^2X_i(x_{st})\ds z^i_{st}+\nabla^2X_0(x_{st})\pd s\right\}
\dt x_{st}\otimes\dt x_{st}
+U^{-1}_{st}\nabla X_i(x_{st})X_i(x_{st})\pd s\pd t. 
\end{align*}
Define 
$$
R_{st}=-\int_0^sU^{-1}_{rt}X_i(x_{rt})d_r z^i_{rt},\q
C_{st}=\int_0^sU^{-1}_{rt}X_i(x_{rt})\otimes U^{-1}_{rt}X_i(x_{rt})dr.
$$
Our calculations show that the $(\cF_{st}:t\ge0)$-semimartingale
$(y_{st}:t\ge0)$ has finite-variation part $(\bar y_{st}:t\ge0)$ and
quadratic variation given by
$$
d_t\bar y_{st}=\tfrac12 R_{st}dt,\q \dt y_{st}\otimes\dt y_{st}=C_{st}dt.
$$
Moreover
$$
d_tx_{st}=U_{st}d_t y_{st}+\tfrac12\dt U_{st}\dt y_{st},
$$
so $(x_{st}:t\ge0)$ has finite-variation part $(\bar x_{st}:t\ge0)$ and
quadratic variation given by
$$
d_t\bar x_{st}=\tfrac12 L_{st}dt,\q \dt x_{st}\otimes\dt x_{st}=\G_{st}dt,
$$
where
\begin{align*}
L_{st}=U_{st}R_{st}&+U_{st}\int_0^sU_{rt}^{-1}
\{\nabla^2X_i(x_{rt})\dr z^i_{rt}+\nabla^2X_0(x_{rt})\pd r\}\G_{rt}\\
&\q\q+U_{st}\int_0^sU_{rt}^{-1}\nabla X_i(x_{rt})X_i(x_{rt})\pd r
\end{align*}
and where $\G_{st}=U_{st}C_{st}U_{st}^*$.

Note that both $(\G_{st}:t\ge0)$ and $(L_{st}:t\ge0)$ are stationary
processes and that, by standard one-parameter estimates,
$\G_{s0}$ and $L_{s0}$ have finite moments of all orders.
By It\^ o's formula, for any $C^2$ function $f$, setting $f_{st}=f(x_{st})$,
the process 
$(f_{st}:t\ge0)$ is an $(\cF_{st}:t\ge0)$-semimartingale
with finite-variation part $(\bar f_{st}:t\ge0)$ and
quadratic variation given by
$$
d_t\bar f_{st}=\tfrac12\left(L_{st}^i\nabla_if(x_{st})+
\G_{st}^{ij}\nabla_i\nabla_jf(x_{st})\right)dt,\q
\dt f_{st}\dt f_{st}=\nabla_i f(x_{st})\G_{st}^{ij}\nabla_j f(x_{st})dt.
$$
In particular, if $m_{st}=f_{st}-f_{s0}-\bar f_{st}$, then $(m_{st}:t\ge0)$
is a (true) martingale. 
Hence, for $f,g\in C^2_b(\R^d)$, we obtain the integration-by-parts
formula
\begin{align*}
&\E[\nabla_i f(x_{s0})\G_{s0}^{ij}\nabla_j g(x_{s0})]
=\lim_{t\da0}\frac1t\E\left[\{f(x_{st})-f(x_{s0})\}\{g(x_{st})-g(x_{s0})\}\right]\\
&\q=-2\lim_{t\da0}\frac1t\E[f(x_{s0})\{g(x_{st})-g(x_{s0})\}]
=-\E[f(x_{s0})\{L_{s0}^i\nabla_ig(x_{s0})+
\G_{s0}^{ij}\nabla_i\nabla_jg(x_{s0})\}].
\end{align*}
An obvious limit argument allows us to deduce the following simple formula,
corresponding to the case $g(x)=x^j$. For all $f\in C^2_b(\R^d)$
and for
$j=1,\dots,d$, we have
$$
\E[\nabla_if(x_{s0})\G_{s0}^{ij}]=-\E[f(x_{s0})L_{s0}^j].
$$
The general formula can then be recovered by replacing $f$ by $f\nabla_jg$
and summing over $j$.

The basic observation underlying this formula is that the distributions
of $(z_0,z_t)$ and $(z_t,z_0)$ are identical,
and hence that the same is true for $(x_{s0},x_{st})$ and $(x_{st},x_{s0})$,
when $(x_{st}:s\ge0)$ is obtained by solving
a stochastic differential equation driven by $(z_{st}:s\ge0)$,
with initial condition independent of $t$.
In fact a stronger notion of reversibility is true.
The distributions of
$(z_{su}:s\ge0,u\in[0,t])$ and $(z_{s,t-u}:s\ge0,u\in[0,t])$ 
are identical, and hence the same is true for 
$(x_{su}:s\ge0,u\in[0,t])$ and $(x_{s,t-u}:s\ge0,u\in[0,t])$.
This may be combined with the fact that the
Stratonovich integral is invariant under time-reversal to see that
$$
\E\left[\{f(x_{st})-f(x_{s0})\}\int_0^tU_{su}^{-1}\du x_{su}\right]
=-2\E\left[f(x_{s0})\int_0^tU_{su}^{-1}\du x_{su}\right].
$$
From this identity, by a similar argument, 
we obtain the following alternative integration-by-parts formula.
For all $f\in C^2_b(\R^d)$, we have
$$
\E[\nabla f(x_{s0})U_{s0}C_{s0}]=-\E[f(x_{s0})R_{s0}].
$$
This formula is the variant discovered by Bismut, which is closely related
to the Clark--Haussmann formula.

\bibliography{pp}

\begin{thebibliography}{10}

\bibitem{MR621660}
Jean-Michel Bismut.
\newblock Martingales, the {M}alliavin calculus and hypoellipticity under
  general {H}\"ormander's conditions.
\newblock {\em Z. Wahrsch. Verw. Gebiete}, 56(4):469--505, 1981.

\bibitem{MR0420845}
R.~Cairoli and John~B. Walsh.
\newblock Stochastic integrals in the plane.
\newblock {\em Acta Math.}, 134:111--183, 1975.

\bibitem{MR972781}
Robert~J. Elliott and Michael Kohlmann.
\newblock Integration by parts, homogeneous chaos expansions and smooth
  densities.
\newblock {\em Ann. Probab.}, 17(1):194--207, 1989.

\bibitem{MR1297021}
K.~D. Elworthy and X.-M. Li.
\newblock Formulae for the derivatives of heat semigroups.
\newblock {\em J. Funct. Anal.}, 125(1):252--286, 1994.

\bibitem{MR2216962}
R{\'e}mi L{\'e}andre.
\newblock The geometry of {B}rownian surfaces.
\newblock {\em Probab. Surv.}, 3:37--88 (electronic), 2006.

\bibitem{MR517243}
Paul Malliavin.
\newblock {$C\sp{k}$}-hypoellipticity with degeneracy.
\newblock In {\em Stochastic analysis (Proc. Internat. Conf., Northwestern
  Univ., Evanston, Ill., 1978)}, pages 199--214. Academic Press, New York,
  1978.

\bibitem{MR517250}
Paul Malliavin.
\newblock {$C\sp{k}$}-hypoellipticity with degeneracy. {II}.
\newblock In {\em Stochastic analysis (Proc. Internat. Conf., Northwestern
  Univ., Evanston, Ill., 1978)}, pages 327--340. Academic Press, New York,
  1978.

\bibitem{MR536013}
Paul Malliavin.
\newblock Stochastic calculus of variation and hypoelliptic operators.
\newblock In {\em Proceedings of the International Symposium on Stochastic
  Differential Equations (Res. Inst. Math. Sci., Kyoto Univ., Kyoto, 1976)},
  pages 195--263, New York, 1978. Wiley.

\bibitem{MR1347353}
J.~R. Norris.
\newblock Twisted sheets.
\newblock {\em J. Funct. Anal.}, 132(2):273--334, 1995.

\bibitem{MR582167}
Ichiro Shigekawa.
\newblock Derivatives of {W}iener functionals and absolute continuity of
  induced measures.
\newblock {\em J. Math. Kyoto Univ.}, 20(2):263--289, 1980.

\bibitem{MR642917}
Daniel~W. Stroock.
\newblock The {M}alliavin calculus, a functional analytic approach.
\newblock {\em J. Funct. Anal.}, 44(2):212--257, 1981.

\bibitem{MR603973}
Daniel~W. Stroock.
\newblock The {M}alliavin calculus and its application to second order
  parabolic differential equations. {I}.
\newblock {\em Math. Systems Theory}, 14(1):25--65, 1981.

\bibitem{MR616961}
Daniel~W. Stroock.
\newblock The {M}alliavin calculus and its application to second order
  parabolic differential equations. {II}.
\newblock {\em Math. Systems Theory}, 14(2):141--171, 1981.

\bibitem{MR0370758}
Eugene Wong and Moshe Zakai.
\newblock Martingales and stochastic integrals for processes with a
  multi-dimensional parameter.
\newblock {\em Z. Wahrscheinlichkeitstheorie und Verw. Gebiete}, 29:109--122,
  1974.

\bibitem{MR0651571}
Eugene Wong and Moshe Zakai.
\newblock Differentiation formulas for stochastic integrals in the plane.
\newblock {\em Stochastic Processes Appl.}, 6(3):339--349, 1977/78.

\end{thebibliography}

\end{document}